\documentclass[11pt]{amsart}

\usepackage{amsmath, amsthm, amssymb}
\usepackage{epsfig}
\usepackage{array}
\usepackage{amsmath}
\usepackage{amsfonts}
\usepackage{amsmath}%
\usepackage{amsfonts}%
\usepackage{amssymb}%
\usepackage{graphicx}

\usepackage{amsfonts}
\usepackage{amssymb}
\usepackage{amsmath,amsxtra,amssymb}
\usepackage{tikz}
\usepackage{hyperref}
\usetikzlibrary{matrix}

\numberwithin{equation}{section}

\newtheorem{theorem}{Theorem}[section]
\newtheorem{lemma}[theorem]{Lemma}
\newtheorem{remark}[theorem]{Remark}
\newtheorem{corollary}[theorem]{Corollary}
\newtheorem{proposition}[theorem]{Proposition}
\newtheorem{definition}[theorem]{Definition}

\newtheorem*{4.1}{Definition 4.1}
\newtheorem*{4.7}{Theoem 4.7}
\newtheorem*{5.2}{Theorem 5.2}
\newtheorem*{5.1}{Proposition 5.1}
\newtheorem*{5.3}{Theorem 5.3}
\newtheorem*{6.3}{Theorem 6.3}
\newtheorem*{7.8}{Corollary 7.8}
\newtheorem*{7.5}{Theorem 7.5}
\newtheorem*{7.7}{Corollary 7.7}

\newcommand{\la}{{\langle\,}}
\newcommand{\ra}{\,\rangle}
\newcommand{\id}{\mathrm{id}}

\newcommand{\tp}{{\widehat{\otimes}}}
\newcommand{\vtp}{{\,\overline{\otimes}\,}}
\newcommand{\btp}{{\,\overline{\boxtimes}\,}}

\newcommand{\h}{{\mathcal H}}

\newcommand{\bd}{{\mathcal{B}}}

\newcommand{\T}{{\mathcal{T}}}

\newcommand{\fee}{{\varphi}}

\newcommand{\s}{\text{span}}

\newcommand{\DD}{{{\ell}^{\infty}(\G)}}
\newcommand{\DDD}{{{\ell}^{\infty}(\hat\G)}}
\newcommand{\DO}{{{\ell}^{1}(\G)}}
\newcommand{\DT}{{{\ell}^{2}(\G)}}

\newcommand{\E}{{\mathcal{E}_\mu}}

\newcommand{\LLL}{{L^{\infty}(\hat\G)}}

\newcommand{\Deltaop}{{\hat{\Delta}}^{\text{{op}}}}

\newcommand{\G}{\mathbb G}

\allowdisplaybreaks

\title[Amenable Actions of Discrete Quantum Groups]
{Amenable Actions of Discrete Quantum Groups on von Neumann Algebras}
\author[M. S. M. Moakhar]{Mohammad S. M. Moakhar}

\address{Department of Mathematics,
        Tarbiat Modares University, Tehran 14115-134, Iran}
\email{m.mojahedi@modares.ac.ir}


\begin{document}

\begin{abstract}
We introduce the notion of Zimmer amenability
for actions of discrete quantum groups on
von Neumann algebras. We prove generalizations of
several fundamental results of the theory in the noncommutative
case. In particular, we give a characterization of Zimmer
amenability of an action $\alpha:\Bbb G\curvearrowright N$
in terms of $\hat{\Bbb{G}}$-injectivity of the von Neumann
algebra crossed product $N\ltimes_\alpha\Bbb G$.
As an application we show that the actions of any
discrete quantum group on its Poisson boundaries are always
amenable.
\end{abstract}

\maketitle

\tableofcontents

\section{Introduction}\label{sect1}

There are many different equivalent 
conditions that characterize 
amenability of a locally compact group $G$.
One such characterization is in terms of a 
fixed point property of affine actions
of $G$.  
In \cite{ZIM}, Zimmer introduced 
the notion of amenable actions as a natural
generalization of this fixed point property. 
In subsequent work, Adams, Elliott and Giordano
characterized Zimmer amenability in terms of the existence
of an equivariant conditional expectation \cite{AEG}. 
In \cite{Del}, 
Delaroche extended Zimmer's definition to
the setting of group actions on von Neumann algebras.
In this paper, we introduce the notion of
Zimmer amenability for
actions of discrete quantum groups
on von Neumann algebras.
\begin{4.1}\label{intro,def}
	Let $\alpha : \G\curvearrowright N$ be an action 
    of a discrete quantum group $\G$ on a von Neumann 
    algebra $N$. Then $\alpha$ is called amenable 
    if there exists a conditional expectation
	$E_\alpha:\DD\vtp N\to \alpha(N)$ such that 
	\[
	(\id \otimes E_\alpha)(\Delta\otimes\id)=
	(\Delta\otimes\id)E_\alpha.
	\]
\end{4.1}
This definition coincides with 
Delaroche's definition \cite[D\'efinition. 3.4]{Del} when 
$G$ is a discrete group. Also observe that
a discrete quantum group is amenable
if and only if its action on the trivial space is amenable
in the above sense. We prove, similarly to the classical result,
the action of every discrete quantum group on itself
by its co-multiplication is amenable (Proposition \ref{itself}).
Moreover, we show a connection between
amenability of discrete quantum groups and amenability of their 
actions on von Neumann algebras which is 
in fact a noncommutative version of 
\cite[Proposition 3.6]{Del}:
\begin{4.7}
	Let $\alpha:\G\curvearrowright N$ be an 
	action of a 
	discrete quantum group $\G$ on a von Neumann algebra
	$N$. The following are equivalent:
	\begin{itemize}
	\item[1.] The quantum group $\G$ is amenable.
	\item[2.] The action $\alpha$ is amenable
	          and there exists an invariant state on $N$.
	\end{itemize}
\end{4.7}
In the case of Kac algebras, this theorem provides
a new characterization for amenability of $\G$ in terms of amenability
of the canonical action of $\G$ on its dual Kac algebra.
\begin{5.2}
    Let $\G$ be a discrete Kac algebra. Then 
    $\G$ is amenable if and only if the 
    canonical action of $\G$ on $\LLL$ is amenable. 
\end{5.2}
One of Zimmer's main motivations to introduce and study
the notion of amenable actions was the applications in the
theory of random walks on $G$ and their associated to
Poisson boundaries of $G$. He proved that for any $G$,
the action of $G$ on its Poisson boundaries
is always amenable \cite[Theorem 5.2]{ZIM}. 
We establish the noncommutative
analogue of this result in the case of 
discrete quantum group actions.
\begin{5.3}
	Let $\G$ be a discrete quantum group and let $\mu\in \DO$
	be a state. The canonical 
	action of $\G$ on the 
	Poisson boundary $\h_\mu$ is amenable.
\end{5.3}
In \cite{ZIM2}, Zimmer studied  
more properties of the amenable action and he characterized 
amenability of the action in terms of injectivity of the
corresponding crossed product \cite[Theorem 2.1]{ZIM2}.
In \cite{Del}, Delaroche generalized this result
to the case of actions on an arbitrary
von Neumann algebra. In fact she proved that
an action $\alpha:G\curvearrowright N$ is amenable if and only
if there exists a conditional expectation from
$B(L^2(G))\vtp N$ onto $N\ltimes_{\alpha}G$
\cite[Proposition 4.1]{Del}. She used this result to show
that amenability of the action on an injective
von Neumann algebra is equivalent to injectivity
of the corresponding crossed product. For discrete
quantum group actions, we will 
characterize Zimmer amenability in terms of
the existence of a conditional expectation that satisfies  
an equivariant condition coming from
induced $\hat\G$ action. More precisely we have
\begin{7.5}
	Let $\alpha:\G\curvearrowright N$ be an action of a 
	discrete quantum group $\G$ on a von Neumann algebra
	$N$. The following are equivalent:
	\begin{itemize}
		\item[1.] The action $\alpha$ is amenable.
		\item[2.] There is an equivariant conditional expectation
		\[
		\hspace{1cm}E:\big(B(\DT)\vtp N, \Deltaop\otimes\id)\to
		\big(N\ltimes_\alpha \G, \hat\alpha\big).
		\]
	\end{itemize}
\end{7.5}
As a direct consequence we will prove a noncommutative analogue of
\cite[Corollaire 4.2]{Del} for the general discrete quantum group actions.
\begin{7.7}
	Let $\alpha:\G\curvearrowright N$ be 
	an action of a discrete quantum group $\G$ 
	on a von Neumann algebra $N$. 
	The following are equivalent:
	\begin{itemize}
		\item [1.] The von Neumann algebra $N$ is 
		           injective and the action $\alpha$ is amenable.
		\item [2.] The crossed product $N\ltimes_\alpha\G$
		           is $\hat\G$-injective.
	\end{itemize}
\end{7.7}
In the case of the trivial action of 
$\G$ on the trivial space, 
Theorem \ref{characterization1} provides a duality
between amenability of $\G$ and injectivity of
the dual von Neumann algebra $\LLL$ in the 
category of $\T(\DT)$-modules where $\T(L_2(\G))$ is
the predual of $B(L_2(\G))$. This perfect duality was
initially investigated by Crann and Neufang in \cite{JN},
(see also \cite{Ja, JN2}).
\newline
Moreover in the case of discrete Kac algebra actions,
we will show that the equivariant condition 
in Theorem \ref{characterization1} can be eliminated.
In fact we have
\begin{6.3}
	Let $\alpha:\G\curvearrowright N$ be an action of a 
	discrete Kac algebra $\G$ on a von Neumann algebra 
	$N$. The following are equivalent:
	\begin{itemize}
	\item [1.] The action $\alpha$ is amenable.
	\item [2.] There is a conditional expectation from 
	           $B(\DT)\vtp N$ onto $N\ltimes_{\alpha}\G$.
	\end{itemize}
\end{6.3}

Beside this introduction, this paper includes six other sections.
In section 2, we recall some notions about discrete quantum groups
and their actions on von Neumann algebras. In section 3, we
construct the von Neumann algebra braided tensor product and 
we use this notion to obtain a version of
diagonal action in the setting of quantum groups. 
In section 4, we introduce 
the notion of amenable actions and we
study some of its properties.
In section 5, we give some examples of amenable actions.
In particular we prove that the action of any discrete 
quantum groups on any of its Poisson boundaries is 
amenable. In section 6, we study actions of discrete 
Kac algebras. The main
result of this section generalize the well-known
fact about the equivalence of amenability
of discrete Kac algebra $\G$ and injectivity of 
$\LLL$. In section 7, we consider the latter 
result in the case of discrete quantum group actions.
\\
\\
\\
\noindent
{\bf{Acknowledgement.}} We are grateful to
Massoud Amini for his continuous encouragement
throughout this project.
We would also like to thank 
Mehrdad Kalantar and Jason Crann for their helpful comments.

\section{Preliminaires}\label{sect2}

In this section we review some basic notions about discrete
quantum groups and their actions on von Neumann algebras.
A {\it discrete quantum group} 
$\G$ is a quadruple $(\DD, \Delta, \varphi, \psi)$, 
where $\DD=\bigoplus_{i\in I}M_{n_i}(\Bbb C)$ 
is a von Neumann algebra 
direct sum of matrix algebras,
$\Delta: \DD\to \DD \vtp \DD$ is a 
co-associative co-multiplication, and 
$\varphi$ and  $\psi$ are normal faithful semi-finite
left, respectively right, invariant weights on $\DD$, that is,
\begin{align*}
	\fee((\omega\otimes\id)\Delta(x))=\omega(1)\fee(x),
	\hspace{.5cm}x\in\mathcal M_\fee,\, \omega\in\DO,\\
    \psi((\id\otimes\omega)\Delta(x))=\omega(1)\psi(x),
	\hspace{.5cm}x\in\mathcal M_\psi,\, \omega\in\DO.
\end{align*}
A discrete quantum group $\G =(\DD, \Delta, \fee, \psi)$
is a \emph{Kac algebra}, if $\fee$ equals $\psi$
and is a trace.

The pre-adjoint of $\Delta$ induces  
an associative completely contractive multiplication 
\[
 *  \, :\,  f \otimes g \,\in\,
\DO\,\tp\, \DO\, \to \,f  *  g\, =\, (f \otimes g)\, \Delta\, \in\, \DO
\]
on $\DO$. Moreover, this maps induces left and right 
actions of $\DO$ on $\DD$ given by:
\begin{equation}\label{bimodule}
\mu *  x:=(\id\otimes\mu)\Delta(x),\hspace{.5cm}
x *  \mu:=(\mu\otimes\id)\Delta(x).
\end{equation}
For a fixed $\mu\in\DO$, the map $x\mapsto x * \mu$
is normal, completely bounded on $\DD$.
This map is
called the \emph{Markov} operator,
if $\mu$ is moreover a state.
A discrete 
quantum group $\G$ is said to be
{\it amenable} if there exists a state 
$m\in \DD^*$ satisfying
\[
\la m , x *  f\ra = \la f , 1\ra \la m , x\ra, \hspace{.5 cm}
 x\in \DD, \,f\in\DO.
\]

The corresponding GNS Hilbert spaces 
$\ell^2(\G,\varphi)$ and $\ell^2(\G,\psi)$ are 
isomorphic and are denoted by the same notation
$\DT$. The (left) fundamental unitary $W$ of $\G$ 
is a unitary operator on $\DT \otimes \DT$, 
satisfying the pentagonal relation
$W_{12}W_{13}W_{23}=W_{23}W_{12},$ in which 
we used the leg notation 
$W_{12}=W\otimes 1, W_{23}=1\otimes W$ and 
$W_{13}=(1\otimes\sigma)W_{12}$, 
where $\sigma(x\otimes y)= y\otimes x$ 
is the flip map on $B(H\otimes K).$ 
The right fundamental unitary $V$ 
with the same properties is 
defined in a similar way
on $\bd(\DT\otimes \DT)$.

Let $\mathcal T(\DT)$ be the predual of $B(\DT)$.
Define the von Neumann algebra $\LLL$ to be the weak*-closure of
$\{(\rho\otimes\id)W: \rho\in \mathcal T(\DT)\}$.
Consider the map $\hat\Delta:\LLL\to\LLL\vtp\LLL$ given by
$\hat\Delta(\hat x)=\hat W^*(1\otimes \hat x)\hat W$,
where $\hat W=\sigma W^*\sigma$. There exists a normal state
$\hat\fee$ on $\LLL$ which is
both invariant of left and right such that the triple
$\hat\G=(\LLL,\hat\Delta, \hat\fee)$ is a \emph{compact}
quantum group called the 
\emph{dual} quantum group of $\G$.

The opposite co-multiplication $\Deltaop$ is given by 
$\Deltaop=\sigma\circ\hat\Delta$. The fundamental unitary 
${\hat{W}}^{\text {op}}$ associated to $\Deltaop$ is defined
by ${\hat{W}}^{\text {op}} = \sigma \hat V \sigma$,
and therefore
${\hat{W}}^{\text {op}}\in \LLL\vtp\DD'$.

The fundamental unitary $W$ of $\G$ induces a co-associative 
co-multiplication on $B(\DT)$ defined by
\[
\Delta^\ell: T\in B(\DT)\mapsto W^*(1\otimes T)W\in B(\DT) \vtp B(\DT).
\]
It is clear that the restriction of $\Delta^\ell$ 
to $\DD$ is the original co-multiplication
$\Delta$ on $\DD$. 
The pre-adjoint of $\Delta^\ell$ induces associative 
completely contractive multiplication on 
the predual $\mathcal T(\DT)$.
\[
*:\omega\otimes \tau \in \T(\DT)\hat{\otimes}\T(\DT)\mapsto
\omega *\tau = \Delta^\ell_*(\omega\otimes \tau)\in \T(\DT).
\]
If $\la \T(\DT)*\T(\DT)\rangle$ denotes the linear span of
$\omega*\tau$ with $\omega,\tau\in\T(\DT)$ we have
\begin{equation}\label{traceclass}
\la \T(\DT)*\T(\DT)\rangle=\T(\DT).
\end{equation}
Similarly to the equations (\ref{bimodule}), there are
left and right actions of 
$\T(\DT)$ on $B(\DT)$.

There is also a co-associative 
co-multiplication on $B(\DT)$ induced by the
right fundamental unitary $V$ which is defined by
\[
\Delta^r: T\in B(\DT)\mapsto V(T\otimes 1)V^*\in B(\DT) \vtp B(\DT).
\]
In a same way, the pre-adjoint of $\Delta^r$ induces associative 
completely contractive multiplication on 
the predual $\mathcal T(\DT)$ with the property 
(\ref{traceclass}).

Let $\G$ be a discrete quantum group. By 
\cite[Proposition 4.2.18]{Ja}, there 
is a conditional expectation $E_0$ from $B(\DT)$ onto
$\DD$ such that 
for any $x\in B(\DT)$ and
$f\in \T(\DT)$, we have 
\begin{align}\label{covariant}
E_0 \big((f\otimes \id)\Delta^\ell(x)\big) = (f\otimes \id)\Delta^\ell(E_0(x)).
\end{align}
In particular, for any $\hat x\in \LLL$, 
$E_0(\hat x)=\hat\varphi(\hat x){\bf 1}$ 
where $\hat\varphi$ is 
the normal invariant state of the compact quantum group $\hat\G$.

A (left) 
action $\alpha : \G\curvearrowright N$
of a discrete quantum group $\G$ on a von Neumann algebra
$N$ is an injective $*$-homomorphism 
$\alpha: N \to \DD\vtp N$ satisfying 
\[
(\Delta\otimes\id)\,\alpha\,=\,(\id\otimes\alpha)\,\alpha\,.
\]
The action of dual quantum group $\hat\G$ is defined similarly.

Let $\alpha : \G\curvearrowright N$ be an action of 
the discrete quantum group $\G$
on the von Neumann algebra $N$. 
A state $\omega$ on $N$ is 
said to be {\it invariant} if 
\[
(f\otimes\omega)\alpha=\la f , 1\ra \omega.
\]
We denote by
$N^\alpha=\{x\in N:\alpha(x)=1\otimes x\}$ 
the {\it fixed point algebra} of the action 
$\alpha : \G\curvearrowright N$.
Let $\theta$ be a normal semi-finite faithful weight on $N$, 
and let $H_\theta$ be the $GNS$ Hilbert space of $\theta$.
It is proved in \cite[Theorem 4.4]{V} that $\alpha$
is implemented by a unitary $U_\alpha\in \DD\vtp B(H_\theta)$, that is,
\begin{equation}\label{1212}
\alpha(x)\, = \,U_\alpha\,(1\otimes x)\,U_\alpha^* \ \ \ (x\in N)\,.
\end{equation}

\begin{definition}\label{1}
	Let $\alpha : \G\curvearrowright N$ and 
	$\beta : \G\curvearrowright M$ be two actions of the discrete
	quantum group $\G$ on von Neumann algebras $N$ and $M$.
	Then a map $\Phi:N\to M$ is equivariant if 
	\[
	(\id\otimes \Phi)\alpha = \beta\circ \Phi.
	\]
	To indicate the actions, we say that 
the map $\Phi:(N, \alpha)\to (M, \beta)$ is equivariant, or that 
$\Phi$ is $(\alpha, \beta)$-equivariant. In the case $\alpha=\beta$,
we say that $\Phi$ is $\alpha$-equivariant.
\end{definition}
The (von Neumann algebra) crossed product of 
the action $\alpha : \G\curvearrowright N$ is
defined by
\[
\G\ltimes_\alpha N :=\{ \alpha(N)\cup(\LLL\vtp{\bf 1})\}''
\subseteq B(\DT)\vtp N.
\]
Analogously to the 
classical setting, there is a characterization 
of the crossed product $\G\ltimes_\alpha N$ 
as the fixed point algebra of a certain action of
$\G$ on $B(\DT)\vtp N$ as follows:

\begin{theorem}[\cite{ENOCK}, Theorem 11.6]\label{characterization}
Let $\alpha : \G\curvearrowright N$ be an action 
of a discrete quantum group $\G$ on a von Neumann 
algebra $N$ and let $\chi$ be the
flip map defined by $\chi(a\otimes b)=b\otimes a$. 
Then there is a left action
$\beta$ on the von Neumann algebra $B(\DT)\vtp N$ defined by
\[
\beta:x\in B(\DT)\vtp N\mapsto 
(\sigma V^*\sigma\otimes 1)\big((\chi\otimes1)(\id\otimes\alpha)(x)\big)
(\sigma V\sigma\otimes 1),
\]
such that
\[
\G\ltimes_\alpha N = (B(\DT)\vtp N)^\beta.
\]
\end{theorem}

If $\alpha : \G\curvearrowright N$ is an action 
of a discrete quantum group $\G$ on a von Neumann 
algebra $N$, there is  is also a natural action $\hat\alpha$ of 
$(\hat\G, \Deltaop)$ on
$\G\ltimes_\alpha N$ which is called the dual action of $\alpha$
 and is defined by
\[
\begin{array}{rll}
\hat\alpha(\alpha(x)) & = & 
1\otimes \alpha(x),\,\,\,\,\, {\text {for all}}\,\, x\in N \\
\hat\alpha(\hat{x}\otimes 1) & = &
\Deltaop(\hat x)\otimes 1,\,\,\,\,\, {\text {for all}}\,\,  \hat x\in \LLL.\\
\end{array}
\]
In fact, we have 
$\alpha(N)=(N\ltimes_\alpha \G)^{\hat\alpha}$ \cite[Theorem 2.7]{V}.

\section{von Neumann algebra braided tensor products}\label{sect3}

In order to some technical obstacles we need to use
a version of diagonal action for discrete quantum group
actions. This section is devoted to a brief introduction
to Yetter--Drinfeld actions and braided tensor products
in von Neumann algebra setting. 
For an overview of these notions,
we refer to \cite{BSV} and \cite{NV}.

Let $\G=(\DD, \Delta, \varphi, \psi)$ be a discrete quantum group.
Consider the triple $(M, \beta, \gamma)$, where $M$ is a
von Neumann algebra on which $\beta$ and $\gamma$ of 
the discrete quantum group $\G$ and 
the dual quantum group $\hat\G$
act. We say $M$ is the $\G$-{\it{YD-algebra}} if the actions 
$\beta$ and $\gamma$ satisfy the following Yetter--Drinfeld condition:
\begin{equation}\label{YD}
(\text{ad}(W)\otimes \id)(\id\otimes\gamma)\beta\,=\,(\sigma\otimes\id)
(\id\otimes\beta)\gamma,
\end{equation}
where $\text{ad}(W)=W\cdot W^*$.

In this case, if $\alpha$ is any action of 
$\G$ on a von Neumann algebra $N$, then 
similarly to \cite[Proposition 8.3]{V2},
we have
\[
\overline{\s\{{\gamma(M)}_{12}{\alpha(N)}_{13}\}}^{\,\,\text {weak*}} \,=\,
\overline{\s\{{\alpha(N)}_{13}{\gamma(M)}_{12}\}}^{\,\,\text {weak*}}.
\]
Hence the weak*-closed linear span of 
$\{{\gamma(a)}_{12}{\alpha(b)}_{13}:a\in M, b\in N\}$ is a von Neumann
subalgebra of $B(\DT)\vtp M\vtp N$, which is called the
{\it{braided tensor product}} of von Neumann algebras $M$ and $N$,
and is denoted by $M\btp N$.
There is a $*$-homomorphism 
$\beta\boxtimes\alpha:M\btp N\to \DD\vtp(M\btp N)$ given by
\[
\beta\boxtimes\alpha(X)=W^*_{12}{U_\beta}_{13} (1\otimes X)
{U_\beta}^*_{13}W_{12},
\]
where the unitary operator $U_\beta$ implements 
the action $\beta$ by (\ref{1212}).\\
In particular, on the set of generators 
$\{{\gamma(M)}_{12}{\alpha(N)}_{13}\}$ we have
\begin{align*}
	(\beta\boxtimes\alpha)({\gamma(a)}_{12}{\alpha(b)}_{13})
&=\, W^*_{12}{U_\beta}_{13}\gamma(a)_{23}\alpha(b)_{24}{U^*_\beta}_{13}W_{12}\\
&=\, W^*_{12}{U_\beta}_{13}\gamma(a)_{23}{U^*_\beta}_{13}\alpha(b)_{24}W_{12}\\
&=\, W^*_{12}(\sigma\otimes\id)({U_\beta}_{23}\gamma(a)_{13}{U^*_\beta}_{23})\alpha(b)_{24}W_{12}\\
&=\, W^*_{12}{\big((\sigma\otimes\id)(\id\otimes\beta)
          \gamma(a)\big)}_{123}{\alpha(b)}_{24}W_{12}\\
&=\, W^*_{12}{\big((\sigma\otimes\id)(\id\otimes\beta)
          \gamma(a)\big)}_{123}W_{12}W^*_{12}\alpha(b)_{24}W_{12}\\
&=\, W^*_{12}{\big((\sigma\otimes\id)(\id\otimes\beta)
          \gamma(a)\big)}_{123}W_{12}
          {\big((\Delta\otimes\id)\alpha(b)\big)}_{124}\\
&=\, W^*_{12}W_{12}{\big((\id\otimes\gamma)\beta(a)\big)}_{123}W^*_{12}W_{12}
          {\big((\Delta\otimes\id)\alpha(b)\big)}_{124}\\
&=\,{\big((\id\otimes\gamma)\beta(a)\big)}_{123}
          {\big((\Delta\otimes\id)\alpha(b)\big)}_{124}\\
&=\,{\big((\id\otimes\gamma)\beta(a)\big)}_{123}
{\big((\id\otimes\alpha)\alpha(b)\big)}_{124}.
\end{align*}
Therefore 
\[
(\beta\boxtimes\alpha)({\gamma(a)}_{12}{\alpha(b)}_{13})=
{\big((\id\otimes\gamma)\beta(a)\big)}_{123}
{\big((\id\otimes\alpha)\alpha(b)\big)}_{124}.
\]
Now it is straightforward to check that the normal
$*$-homomorphism $\beta\boxtimes\alpha$
is in fact an action 
of the discrete quantum group $\G$ on 
the von Neumann algebra $M\btp N$.

If $L$ and $M$ are $\G$-YD-algebras and $N$ is a von Neumann
algebra on which $\G$ acts, then similarly to \cite{NV},
we can construct the braided tensor products
$(L\btp M)\btp N$ and $L\btp (M\btp N)$ and there is a
natural identification 
\begin{equation}\label{associative}
(L\btp M)\btp N\cong L\btp (M\btp N).
\end{equation}

For any discrete quantum group $\G$, there is
an action $\gamma:\hat\G\curvearrowright \DD$ given by 
\[
\gamma(x)=\hat{W}^*(1\otimes x)\hat{W}.
\]
Observe that 
\[
(\text{ad}(W)\otimes \id)(\id\otimes\gamma)\Delta(x)=(\sigma\otimes\id)
(\id\otimes\Delta)\gamma(x).
\]
It implies that the pair $(\Delta, \gamma)$ satisfies
the compatibility condition (\ref{YD}) and therefore 
$\DD$ is a $\G$-YD-algebra. In this 
paper, we always consider braided tensor products
whose first legs are $\DD$.

The following is the von Neumann algebraic version 
of \cite[Lemma 1.24]{BSV}. 
We included the proof for the
convenience of the reader.
\begin{lemma}\label{braided product}
Let $\alpha:\G\curvearrowright N$ be an action of a discrete quantum 
group $\G$ on a von Neumann algebra $N$. There exists an equivarinat
$*$-isomorphism 
\[
T_\alpha:(\DD\btp N , \Delta\boxtimes \alpha)\to
 (\DD\vtp N , \Delta\otimes \id)
\]
such that $T_\alpha( 1\boxtimes a)=\alpha(a)$ for all $a\in N$ and
$T_\alpha(x\boxtimes 1)=x\otimes 1$ for all $x\in \DD$.
\end{lemma}
\begin{proof}
It is sufficient to define $T_\alpha$ on
the set of generators $\{{\gamma(a)}_{12}{\alpha(b)}_{13}\}$. For
all $a\in\DD$ and $b\in N$
\[
T_\alpha({\gamma(a)}_{12}{\alpha(b)}_{13}):=
(\id\otimes\alpha^{-1})\Big((\sigma\otimes \id)\big(
(\sigma W^*\sigma\otimes 1)({\gamma(a)}_{12}{\alpha(b)}_{13})
(\sigma W\sigma\otimes 1)\big)\Big).
\]
Then the map $T_\alpha:\DD\btp N \to \DD\vtp N$ is well-defined.
Indeed, for $a\in\DD$ and $b\in N$ we have
\begin{align*}
(\sigma W^*\sigma\otimes 1)&({\gamma(a)}_{12}{\alpha(b)}_{13})
(\sigma W\sigma\otimes 1)\\
&=\,(\sigma W^*\sigma\otimes 1)\big((\hat{W}^*\otimes 1)
(1\otimes a\otimes 1)(\hat{W}\otimes 1)
{\alpha(b)}_{13}\big)(\sigma W\sigma\otimes 1)\\
&=\,(\hat{W}\otimes 1)\big((\hat{W}^*\otimes 1)
(1\otimes a\otimes 1)(\sigma W^*\sigma\otimes 1)
{\alpha(b)}_{13}\big)(\sigma W\sigma\otimes 1)\\
&=\,(1\otimes a\otimes 1)(\sigma W^*\sigma\otimes 1)
{\alpha(b)}_{13}(\sigma W\sigma\otimes 1)\\
&=\,(1\otimes a\otimes 1)(\sigma\otimes\id)
\big(W^*_{12}{\alpha(b)}_{23}W_{12}\big)\\
&=\,(\sigma\otimes\id)
\big((a\otimes 1\otimes 1)W^*_{12}{\alpha(b)}_{23}W_{12}\big)\\
&=\,(\sigma\otimes\id)\big((a\otimes 1\otimes 1)
(\Delta\otimes \id)(\alpha(b))\big).
\end{align*}
Therefore by definition of $T_\alpha$ we have
\[
	T_\alpha({\gamma(a)}_{12}{\alpha(b)}_{13})=(\id\otimes\alpha^{-1})
	\big((a\otimes 1\otimes 1)(\id\otimes \alpha)(\alpha(b))\big)
	=(a\otimes 1)\alpha(b).
\]
Since the linear span of $\{(\DD\otimes 1)\alpha(N)\}$ is weak* dense in
$\DD\vtp N$, $T_\alpha$ is a $*$-isomorphism from $\DD\btp N$ onto
$\DD\vtp N$ and it is clear that
for all $a\in N$ and $x\in\DD$,
$T_\alpha( 1\boxtimes a)=\alpha(a)$ and
$T_\alpha(x\boxtimes 1)=x\otimes 1$.
\end{proof}

\section{Amenable actions}\label{sect4}

In this section, we introduce the notion of 
amenable action of discrete quantum groups on
von Neumann algebras. This definition
is a generalization of the amenable action of 
discrete groups on von Neumann algebras
introduced in \cite[D\'efinition 3.4]{Del}. 
Recall that the homomorphism $\alpha:G\to {\text{Aut}}(M)$ is
called an {\it action} of a discrete group $G$ on 
a von Neumann algebra 
$M$. If $\tau$ denotes the left translation action of $G$
on $\ell^\infty(G)$, then the 
action $\alpha$ is called {\it amenable} if
there exists an equivariant conditional expectation 
$P: (\ell^\infty(G)\vtp M, \tau\otimes\alpha)\to({\bf 1}\vtp M, \alpha)$, 
i.e.,
\[
P(\tau_g\otimes\alpha_g)=\alpha_g\circ P,\,\,\,\,\,\,\,\,\, g\in G.
\]
There exists an automorphism 
$T_\alpha$ on $\ell^\infty(G)\vtp M$ defined by 
\[
T_\alpha\big(\sum_{g\in G}(\delta_g \otimes x_g)\big)=
\sum_{g\in G}(\delta_g \otimes \alpha_g^{-1}(x_g)).
\]
Since $\alpha(x)=\sum_{g\in G}(\delta_g \otimes \alpha_g(x))$ 
for all $x\in M$, we have 
$T_\alpha({\bf 1}\vtp M)=\alpha(M)$.
In some sense this means that the automorphism $T_\alpha$
make it possible to get away with the ``twisting'' effect
of $\alpha$. 
It is straightforward to check that 
\[
(\tau_g\otimes\id)\circ T_\alpha =
T_\alpha\circ(\tau_g\otimes\alpha_g),
\]
for all $g\in G$. 
So $T_\alpha$ is an equivariant isomorphism from
$(\ell^\infty(G)\vtp M, \tau\otimes\alpha)$ onto
$(\ell^\infty(G)\vtp M, \tau\otimes\id)$.

In summary, we have the following commutative diagram for 
the amenable action $\alpha$ of a discrete 
group $G$ on a von Neumann algebra $N$:
\begin{equation}\label{diagram}
\begin{tikzpicture}
  \matrix (m) [matrix of math nodes,row sep=4em,column sep=8em,minimum width=2em]
  {
    (\ell^\infty(G)\vtp M,\tau\otimes \alpha) & (\ell^\infty(G)\vtp M,\tau\otimes\id)\\
    ({\bf 1}\vtp M,\tau\otimes \alpha) & (\alpha(M),\tau\otimes\id)\\};
  \path[-stealth]
    (m-1-1) edge node [left] {$P$} (m-2-1)
            edge node [above] {$T_\alpha$} (m-1-2)
    (m-2-1.east|-m-2-2) edge node [below] {$T_\alpha$} 
            node [above] {} (m-2-2)
    (m-1-2) edge node [right] {$\overline{P}$} (m-2-2);
\end{tikzpicture}
\end{equation}

This diagram allows us to define an equivalent 
definition for the amenable action of
discrete groups on von Neumann algebras.
Let $\alpha:G\to {\text{Aut}}(M)$ be an action 
of a discrete group $G$ on a von Neumann 
algebra $M$
and let $\tau$ be the left translation 
action on $\ell^\infty(G)$.
Then the action  $\alpha$ is called amenable if 
there exists an equivariant conditional expectation
\[
P:(\ell^\infty(G)\vtp M,\tau\otimes\id)\to 
(\alpha(M),\tau\otimes\id).
\]
Motivated by this definition, 
we introduce the notion of the amenable 
action of discrete quantum 
groups on von Neumann algebras.

\begin{definition}\label{amenability}
	Let $\alpha : \G\curvearrowright N$ be an action 
    of a discrete quantum group $\G$ on a von Neumann 
    algebra $N$. Then $\alpha$ is called amenable 
    if there exists a conditional expectation
	$E_\alpha:\DD\vtp N\to \alpha(N)$ such that 
	\begin{equation}\label{G-equivariant}
	(\id \otimes E_\alpha)(\Delta\otimes\id)=
	(\Delta\otimes\id)E_\alpha.
	\end{equation}
\end{definition}

\begin{remark}
	The diagram (\ref{diagram}) shows that the 
	Definition \ref{amenability} coincides with 
	the classical definition of amenable actions 
	introduced in \cite{Del}.
\end{remark}

\begin{remark}\label{remark}
   The trivial action $\text{tr}:\G\curvearrowright\Bbb C$
   of a discrete quantum group $\G$ on the trivial space 
   is amenable if and only if 
   $\G$ is amenable. Indeed, if the trivial action $\text{{tr}}$
   is amenable, then there is an 
   equivariant conditional expectation 
   $E_{\text{tr}}:(\DD\vtp \Bbb C, \Delta\otimes \id)\to
   (\Bbb C\otimes 1,\Delta\otimes \id)$. 
   Define a state $m$ on $\DD$ by 
   $E_{\text{tr}}(x\otimes 1)=m(x)1\otimes 1$. Then
\begin{align*}
m(x  *  f)1\otimes 1 & =\,E_{\text{tr}}\big((x *  f)\otimes 1\big)\\
&=\,E_{\text{tr}}\big((f\otimes \id \otimes \id)
(\Delta\otimes \id)(x\otimes 1)\big)\\
&=\,(f\otimes \id \otimes \id)(\id\otimes E_{\text{tr}})
(\Delta\otimes \id)(x\otimes 1)\\
&=\,(f\otimes \id \otimes \id)(\Delta\otimes \id)
E_{\text{tr}}(x\otimes 1)\\
&=\,m(x)(f\otimes \id \otimes \id)(\Delta\otimes \id)(1\otimes 1)\\ 
&=\,m(x)f(1) 1\otimes 1.
\end{align*}
   Therefore $m$ is an invariant 
   mean on $\DD$. For the converse, if $\G$ is amenable,
   there exists an invariant state $m$ on $\DD$.
   Define the conditional expectation 
   $P:\DD\vtp \Bbb C\to\Bbb C\vtp\Bbb C$ by $P=m\otimes\id$.
   It is easy to check that $P$ is $(\Delta\otimes\id)$-equivariant.
   (See also Theorem \ref{relation}.)
\end{remark}

\begin{definition}
	Let $\alpha:\G\curvearrowright M$ and $\beta:\G\curvearrowright N$ 
	be actions of a discrete quantum group $\G$ on
	von Neumann algebras $M$ and $N$, respectively, 
	where $M$ is a von Neumann
	subalgebra of $N$. Then
	\begin{itemize}
		\item [1.] The triple $(N, \G, \beta)$ is an extension of
		           $(M, \G, \alpha)$ if $\alpha$ is the 
		           restriction of $\beta$ to $M$ and there 
		           is a conditional expectation
		           from $M$ onto $N$.
		\item [2.] For the extension $(N, \G, \beta)$
	               of $(M, \G, \alpha)$, the pair $(N, M)$ is 
	               called amenable, if there is an equivariant
	               conditional expectation $P$ from 
	               $(N, \beta)$ onto $(M, \alpha)$.
	\end{itemize}
\end{definition}

\begin{proposition}\label{extension}
   Let $(N, \G, \beta)$ be an extension of $(M, \G, \alpha)$.
	\begin{itemize}
		\item [1.] If the action $\alpha$ is amenable, then
		           the pair $(N, M)$ is amenable.
		\item [2.] If the action $\beta$ is amenable and the 
		           pair $(N, M)$ is amenable, then the action
		           $\alpha$ is amenable.
	\end{itemize}
\end{proposition}

\begin{proof}
	(1): Assume that $(N, \G, \beta)$ is an extension of $(M, \G, \alpha)$
	and therefore there is a conditional expectation $Q$ from 
	$N$ onto $M$. Since $\alpha$
	is amenable we have an equivariant conditional expectation $E_\alpha$ 
	from $(\DD\vtp M, \Delta\otimes\id)$ onto 
	$(\alpha(M), \Delta\otimes\id)$.
	Define the conditional expectation $P:N\to M$ by
	\[
	P = \alpha^{-1}\circ E_\alpha \circ (\id\otimes Q)\circ \beta.
	\]
	Then we have
	\begin{align*}
(\id\otimes P)\beta & =\, 
(\id \otimes \alpha^{-1}\circ E_\alpha)(\id \otimes Q)
(\id \otimes \beta)\beta\\
&=\,(\id \otimes \alpha^{-1})(\id\otimes E_\alpha)
(\id \otimes Q)(\Delta\otimes \id)\beta\\
&=\,(\id \otimes \alpha^{-1})(\id\otimes E_\alpha)
(\Delta\otimes \id)(\id \otimes Q)\beta\\
&=\,(\id \otimes \alpha^{-1})(\Delta\otimes \id)E_\alpha(\id \otimes Q)\beta\\
&=\,(\id \otimes \alpha^{-1})(\id \otimes\alpha)E_\alpha(\id \otimes Q)\beta\\
&=\,\alpha \circ P.
   \end{align*}
It shows that $P:(N,\beta)\to (M, \alpha)$ is equivariant.
\newline
	(2): Now suppose that $\beta$ is amenable, then 
	there is an equivariant conditional expectation $E_\beta$
	from $(\DD\vtp N, \Delta\otimes\id)$ onto 
	$(\beta(N), \Delta\otimes\id)$. Since 
	the pair $(N, M)$ is amenable,
	there is also a conditional expectation $P$ from $N$ onto
	$M$ such that $(\id\otimes P)\beta=\alpha\circ P$. Hence
	the composition $(\id\otimes P)E:\DD\vtp N\to \alpha(M)$ is
	a conditional expectation such that
	\begin{align*}
	(\id\otimes\id\otimes P)(\id\otimes E)(\Delta\otimes\id) &=\,
	(\id\otimes\id\otimes P)(\Delta\otimes\id)E\\
	&=\,(\Delta\otimes\id)(\id\otimes P)E.
	\end{align*}
	Since $\alpha(M)\subseteq \DD\vtp M$,
	by restricting of $(\id\otimes P)E$ to 
	$\DD\vtp M$ we obtain an
	equivariant conditional expectation from
	$(\DD\vtp M, \Delta\otimes\id)$ onto $(\alpha(M), \Delta\otimes\id)$,
	and therefore the action $\alpha$ is amenable.
\end{proof}
    Let $\alpha:\G\curvearrowright N$ be an action of a 
	discrete quantum group $\G$ on a
	von Neumann algebra $N$. Consider the conditional expectation $E_0$ as 
    the equation (\ref{covariant})
    and fix an arbitrary state $f\in\DO$. Then 
    $E_0\otimes f \otimes \id$ is 
    a conditional expectation from 
    $B(\DT)\vtp\DD\vtp N$ onto 
    $B(\DT)\vtp {\bf 1}\vtp N$. By
    restricting we obtain a conditional 
    expectation from $\DD\btp N$ onto
    ${\bf 1}\btp N$. So the triple
    $(\DD\btp N, \Delta\boxtimes \alpha, \G)$ is an extension of
	$({\bf 1}\btp N, \Delta\boxtimes \alpha, \G)$.
\begin{proposition}\label{amenable pair}
	Let $\alpha:\G\curvearrowright N$ be an action of a 
	discrete quantum group $\G$ on a
	von Neumann algebra $N$. Then the action $\alpha$
	is amenable if and only if for the extension
	$(\DD\btp N, \Delta\boxtimes \alpha, \G)$
	of $({\bf 1}\btp N, \Delta\boxtimes \alpha, \G)$, the pair
	$(\DD\btp N, {\bf 1}\btp N)$ is amenable.
\end{proposition}

\begin{proof}
    By Lemma \ref{braided product}, there exists an
    equivariant $*$-isomorphism between 
    $(\DD\btp N , \Delta\boxtimes \alpha)$ and
    $(\DD\vtp N , \Delta\otimes \id)$. Since
    $(\DD\btp N, \Delta\boxtimes \alpha, \G)$ is an extension of
	$({\bf 1}\btp N, \Delta\boxtimes \alpha, \G)$, the action $\alpha$
	is amenable if and only if the pair
	$(\DD\btp N, {\bf 1}\btp N)$ is amenable.
\end{proof}

The following result is a noncommutative version of 
\cite[Proposition 3.6]{Del}.
\begin{theorem}\label{relation}
	Let $\alpha:\G\curvearrowright N$ be an 
	action of a 
	discrete quantum group $\G$ on a von Neumann algebra
	$N$. The following are equivalent:
	\begin{itemize}
	\item[1.] The quantum group $\G$ is amenable.
	\item[2.] The action $\alpha$ is amenable
	          and there exists an invariant state on $N$.
	\end{itemize}
\end{theorem}
\begin{proof}
    (2) $\Rightarrow$ (1): suppose that $\omega$ 
	is an invariant state 
	on $N$ and $E_\alpha$ is an equivariant 
	conditional expectation from $\DD\vtp N$ onto $\alpha(N)$
	coming from amenability of the action $\alpha$.
	Define a state $m$ on $\DD$ by
	\[
	\la m , x \ra = \la \omega , \alpha^{-1}\circ E_\alpha(x\otimes 1)\ra.
	\]
	Then $m$ is a left invariant state on $\DD$. Indeed, for any 
	$x\in \DD$ and
	$f\in\DO$ we have 
	 \begin{align*}
	 \la m , x *  f \ra &=\,\la m , (f\otimes\id)\Delta(x) \ra\\
&=\,\la \omega , \alpha^{-1}\circ 
E_\alpha\big(((f\otimes\id)\Delta(x))\otimes 1\big) \ra\\
&=\,\la \omega , \alpha^{-1}\big((f\otimes\id\otimes\id)(\id\otimes E_\alpha)
(\Delta\otimes \id)(x \otimes 1)\big) \ra\\
&=\,\la \omega , \alpha^{-1}\big((f\otimes\id\otimes\id)(\Delta\otimes \id)
E_\alpha(x \otimes 1)\big) \ra\\
&=\,\la \omega , \alpha^{-1}\big((f\otimes\id\otimes\id)(\id\otimes \alpha)
E_\alpha(x \otimes 1)\big) \ra\\
&=\,\la \omega , (f\otimes\id)E_\alpha(x \otimes 1) \ra\\
&=\,f(1)\la \omega , \alpha^{-1}\circ E_\alpha(x \otimes 1) \ra\\
&=\,f(1)\la m , x \ra,
\end{align*} 
	where we use the fact that $\omega$ is invariant 
	in the penultimate step.

	\noindent
	(1) $\Rightarrow $ (2): suppose that $m$ is a 
	left invariant mean on $\DD$.
	Fix $\eta\in N_*$ and define $\omega:=(m\otimes\eta)\alpha$. 
	Then for any
	$f\in \DO$ and $x\in N$, we have
	 \begin{align*}
	 \la f\otimes\omega , \alpha(x) \ra & =\,\la \omega , (f\otimes\id)\alpha(x) \ra\\
&\,=\la m\otimes\eta , (f\otimes\id\otimes\id)(\id\otimes \alpha)\alpha(x) \ra\\
&\,=\la m\otimes\eta , (f\otimes\id\otimes\id)(\Delta\otimes \id)\alpha(x) \ra\\
&\,=\la m , (f\otimes\id)\Delta\big((\id\otimes\eta)\alpha(x)\big) \ra\\
&\,=f(1)\la m , (\id\otimes\eta)\alpha(x) \ra\\
&\,=f(1)\la m\otimes\eta , \alpha(x) \ra\\
&\,=f(1)\la \omega , x \ra,
	 \end{align*}
	it shows $\omega$ is an invariant state on $N$.
	\newline
	Now we prove that the action $\alpha$ is amenable.
	First, we claim that for any $x\in\DD\btp N$, we have
	    \begin{equation}\label{new equality}
    	(\Delta\boxtimes \alpha)(x)=(\text{ad}(W^*)\otimes\id\otimes\id)(\sigma\otimes\id\otimes\id)(\id\otimes\Delta\otimes\id)(x).
    \end{equation}
	Since the linear span of $\{\gamma(a_i)_{12}\alpha(b_i)_{13}\}$ is 
	weak* dense in $\DD\btp N$, we need to show
	(\ref{new equality}) for the set of generators. For any
	$a\in \DD$ and $b\in N$ we have
\begin{align*}
    (\Delta\boxtimes \alpha)\big(&{\gamma (a)}_{12}{\alpha(b)}_{13}\big)	
    \,=\,\big((\id\otimes\gamma)\Delta(a)\big)_{123}\big((\id\otimes\alpha)\alpha(b)\big)_{124}\\
    &=\,\big((\text{ad}(W^*)\otimes\id)(\sigma\otimes\id)(\id\otimes\Delta)\gamma(a)\big)_{123} \big((\id\otimes\alpha)\alpha(b)\big)_{124}\\
    &=\,\big(W^*_{12}(\sigma\otimes\id)\big((\id\otimes\Delta)\gamma(a)\big) W_{12}\big)_{123} \big((\id\otimes\alpha)\alpha(b)\big)_{124}\\    
    &=\,\big(W^*_{12}(\sigma\otimes\id)\big((\id\otimes\Delta)\gamma(a)\big)W_{12}\big)_{123} \big((\Delta\otimes\id)\alpha(b)\big)_{124}\\    
    &=\,\big(W^*_{12}(\sigma\otimes\id)\big((\id\otimes\Delta)\gamma(a)\big) W_{12}\big)_{123}
        W^*_{12}\alpha(b)_{24}W_{12}\\    
    &=\,W^*_{12}\big((\sigma\otimes\id)(\id\otimes\Delta)\gamma(a)\big)_{123}\alpha(b)_{24}W_{12}\\
    &=\,(\text{ad}(W^*)\otimes\id\otimes\id)(\sigma\otimes\id\otimes\id)\big[\big((\id\otimes\Delta)\gamma(a)
        \big)_{123}\,\alpha(b)_{14}\big]\\
    &=\,(\text{ad}(W^*)\otimes\id\otimes\id)(\sigma\otimes\id\otimes\id)(\id\otimes\Delta\otimes\id)(\gamma(a)_{12}\,\alpha(b)_{13}).
    \end{align*}
    Hence we conclude the equality (\ref{new equality}).
	Let $E_0:B(\DT)\to \DD$ be the normal conditional 
	expectation given by (\ref{covariant}).
	Therefore for any $\hat{x}\in\LLL$ and $f,\omega\in \DO$, we have
	 \begin{align*}
\la \omega\otimes f , (\id\otimes E_0)\Delta^\ell(\hat{x}) \rangle &=\,
\la f , E_0\big((\omega\otimes\id)\Delta^\ell(\hat{x})\big) \rangle \\
&=\,\la f ,(\omega\otimes\id)\Delta^\ell(E_0(\hat{x}))\rangle \\
&=\,\la \omega\otimes f , \hat{\varphi}(\hat{x}) 1\otimes 1 \rangle.
	 \end{align*}
Hence 
\begin{equation}\label{E_0}
(\id\otimes E_0)\Delta^\ell(\hat{x}) = \hat{\varphi}(\hat{x}) 1\otimes 1,
\end{equation} 
    for all $\hat{x}\in\LLL$. Consider the conditional expectation
	$E_0\otimes m\otimes\id$ from the von Neumann algebra
	$B(\DT)\vtp\DD\vtp N$ onto $\DD\vtp{\bf 1}\vtp N$.
	Then by restricting, there is a conditional expectation
	$E$ from $\DD\btp N$ onto ${\bf 1}\btp N$.
	We show that the conditional expectation $E$ is
	$(\Delta\boxtimes\alpha)$-equivariant. For any 
	$a\in\DD$ and $b\in N$, by the equality (\ref{E_0}) we have
	\begin{align*}
		(\id\otimes E_0\otimes \id &\otimes\id)
		\Big[W^*_{12}{\gamma(a)}_{23}\,{\alpha(b)}_{24}
        W_{12}\Big] \\
    &=\,(\id\otimes E_0\otimes \id\otimes\id)
        \Big[W^*_{12}{\gamma(a)}_{23}W_{12}W^*_{12}{\alpha(b)}_{24}
        W_{12}\Big]\\
    &=\,(\id\otimes E_0\otimes\id\otimes\id)\Big[
        \big((\Delta^\ell\otimes\id)\gamma(a)\big)_{123}
        \big((\id\otimes\alpha)\alpha(b)\big)_{124}\Big]\\  
    &=\,\big((\id\otimes E_0\otimes \id)
        (\Delta^\ell\otimes\id)\gamma(a)\big)_{123}
        \big((\id\otimes\alpha)\alpha(b)\big)_{124}\\  
    &=\,(\id\otimes \hat\fee\otimes \id\otimes \id)\big(\gamma(a)_{23}\big)
        \big((\id\otimes\alpha)\alpha(b)\big)_{124}. 
	\end{align*}
	Consider $x\in \DD\btp N$ as 
    $\displaystyle\sum_{i\in I}{{\gamma(a_i)}_{12}{\alpha(b_i)}_{13}}$.
    Since $m$ is an invariant mean, 
    the equality (\ref{new equality}) and the above 
    calculation yield that
    \begin{align*}
   	&(\id\otimes E)(\Delta\boxtimes \alpha)(x)\\
   	&=\,(\id\otimes E)\Big[(\text{ad}(W^*)\otimes\id\otimes\id)(\sigma\otimes\id\otimes\id)(\id\otimes\Delta\otimes\id)(x)\Big]\\
    &=\,(\id\otimes E)\Big[W^*_{12}\big((\sigma\otimes\id\otimes\id)(\id\otimes\Delta\otimes\id)(x)\big)W_{12}\Big]\\
    &=\,(\id\otimes E_0\otimes m\otimes\id)\Big[W^*_{12}\big((\sigma\otimes\id\otimes\id)(\id\otimes\Delta\otimes\id)(x)\big)W_{12}\Big]\\
    &=\,(\id\otimes E_0\otimes \id)\Big[W^*_{12}\big((\sigma\otimes\id)(\id\otimes\id\otimes m \otimes\id)(\id\otimes\Delta\otimes\id)(x)\big)W_{12}\Big]\\
    &=\,(\id\otimes E_0\otimes \id)\Big[W^*_{12}\big((\sigma\otimes\id)(\id\otimes\id\otimes m \otimes\id)(x_{134})\big)W_{12}\Big]\\
    &=\,(\id\otimes E_0\otimes m\otimes\id)
        \Big[W^*_{12}(1\otimes x)W_{12}\Big]\\  
    &=\,(\id\otimes E_0\otimes m\otimes\id)\Big[W^*_{12}
        \displaystyle\sum_{i\in I}{{\gamma(a_i)}_{23}{\alpha(b_i)}_{24}}\,
        W_{12}\Big]\\
    &=\,(\id\otimes \id\otimes m\otimes\id)\Big[\displaystyle\sum_{i\in I}
        (\id\otimes \hat\fee\otimes \id\otimes \id)\big(\gamma(a_i)_{23}\big)
        \big((\id\otimes\alpha)\alpha(b_i)\big)_{124}\Big],                                      \end{align*}
    where we use the normality of the conditional expectation $E_0$
    in the last equality.
    On the other hand, repeating the calculation show that
    \begin{align*}
   	(\Delta\boxtimes &\alpha)E(x)=
   	(\Delta\boxtimes \alpha)(E_0\otimes m\otimes\id)(x)\\
   	&=\,(\Delta\boxtimes \alpha)(E_0\otimes m\otimes\id)
   	\big[\displaystyle\sum_{i\in I}
   	{{\gamma(a_i)}_{12}\,{\alpha(b_i)}_{13}}\big]\\
   	&=\,(\Delta\boxtimes \alpha)(\id\otimes m\otimes\id)
   	\Big[\displaystyle\sum_{i\in I}(E_0\otimes \id\otimes\id)\big[
   	{{\gamma(a_i)}_{12}\,{\alpha(b_i)}_{13}}\big]\Big]\\
   	&=\,(\Delta\boxtimes \alpha)(\id\otimes m\otimes\id)
   	\Big[\displaystyle\sum_{i\in I}(E_0\otimes \id\otimes\id)\big[
   	{{\gamma(a_i)}_{12}\big]{\alpha(b_i)}_{13}}\Big]\\
   	&=\,(\Delta\boxtimes \alpha)(\id\otimes m\otimes\id)
   	\Big[\displaystyle\sum_{i\in I}(\hat\fee\otimes \id\otimes\id)\big[
   	{{\gamma(a_i)}_{12}\big]\,{\alpha(b_i)}_{13}}\Big]\\
   	&=\,(\id\otimes\id\otimes m\otimes\id)
   	\Big[\displaystyle\sum_{i\in I}(\id\otimes\hat\fee\otimes \id\otimes\id)
   	\big({\gamma(a_i)}_{23}\big)(\id\otimes\alpha)\alpha(b_i)\big)_{124}\Big]
   	\end{align*}
   	From these two calculations, it follows that 
   	the pair $(\DD\btp N,{\bf 1}\btp N)$ is amenable and 
   	by Proposition \ref{amenable pair} the action $\alpha$ is amenable.
\end{proof}
\begin{theorem}\label{isomorphism}
	Let $\alpha:\G\curvearrowright N$ be an action of 
	a discrete quantum group $\G$ on a von Neumann algebra
	$N$. Then there is an equivariant 
	isomorphism $\Phi$ from 
	$\big((\DD\btp N)\ltimes_{\Delta\boxtimes\alpha} \G,
	\widehat{\Delta\boxtimes\alpha} \big)$ onto 
	$\big(B(\DT)\vtp N, \widehat{\Delta\boxtimes\alpha} \big)$
	 such that $\Phi$ maps
	$({\bf 1}\btp N)\ltimes_{\Delta\boxtimes \alpha} \G$
	onto $N\ltimes_\alpha \G$.
\end{theorem}

\begin{proof}
	Consider the equivariant $*$-isomorphism $T_\alpha$,
	given by Lemma \ref{braided product},  
	from $(\DD\btp N, \Delta\boxtimes\alpha)$ onto
	$(\DD\vtp N, \Delta \otimes \id)$.
	Then the isomorphism
	$\Phi$ is obtained from the identification:
	 \begin{align*}
(\DD\btp N)&\ltimes_{\Delta\boxtimes\alpha}\G
\,=\,[\,(\Delta\boxtimes\alpha)(\DD\btp N)
\cup (\LLL\vtp {\bf 1}_{\DD\btp N})\,]''\\
&\,\cong
[\,(\id\otimes T_\alpha)(\Delta\boxtimes\alpha)(\DD\btp N)\cup (\LLL\vtp {\bf 1}_{\DD\vtp N})\,]''\\
&\,=
[\,(\Delta\otimes \id)T_\alpha(\DD\btp N)\cup
 (\LLL\vtp {\bf 1}\vtp {\bf 1})V_{12}V^*_{12}\,]''\\
&\,=
[\,(\Delta\otimes \id)T_\alpha(\DD\btp N)\cup
 V_{12}(\LLL\vtp {\bf 1}\vtp {\bf 1})V^*_{12}\,]''\\
&\,=
[\,(\Delta\otimes \id)T_\alpha(\DD\btp N)\cup
 (\Delta\otimes \id)(\LLL\vtp {\bf 1})\,]''\\
&\,\cong
[\,T_\alpha(\DD\btp N)\cup (\LLL\vtp {\bf 1})\,]''\\
&\,\cong 
[(\DD\vtp N)\cup (\LLL\vtp {\bf 1})\,]''\\
&\,=
B(\DT)\vtp N
     \end{align*}
    where in the fourth equality, we used the fact that 
    $V\in \LLL'\vtp \DD$.
    In particular, for any $x\in\DD\btp N$ and any 
    $\hat x\in \LLL$ we have
    \begin{equation}\label{Phi2}
    	    \Phi\big((\Delta\boxtimes\alpha)(x)\big)=T_\alpha(x), 
    \hspace{0.5cm}\Phi(\hat x\otimes 1_{\DD\btp N})=
    \hat x\otimes 1_{\DD\vtp N}.
    \end{equation}
	From Lemma \ref{braided product}, we know that 
	$T_\alpha({\bf 1}\boxtimes N)=\alpha(N)$, and therefore 
	by the same calculations we have 
	\[ 
({\bf 1}\btp N)\ltimes_{\Delta\boxtimes\alpha} \G \cong N\ltimes_\alpha \G.
	\]
	In order to show the equivariant condition, 
	it is sufficient to check the
	equality on the set of generators of
	$(\DD\btp N)\ltimes_{\Delta\boxtimes\alpha}\G$. Suppose that
	$x\in\DD\btp N$, then since 
	${\hat{W}}^{\text{op}}\in \LLL\vtp\DD'$, by
	(\ref{Phi2}) we have
	\begin{align*}	 
\big((\Deltaop\otimes\id)\circ \Phi\big)(\Delta\boxtimes\alpha)(x)
& =\,(\Deltaop\otimes\id)(\Phi(\Delta\boxtimes\alpha)(x))\\
& =\,(\Deltaop\otimes\id)T_\alpha(x)\\
& =\,{\hat{W}}^{{\text{op}}^*}_{12}\, (1 \otimes T_\alpha(x)) \,{\hat{W}}^{\text{op}}_{12}\\
& =\,1 \otimes T_\alpha(x)\\
& =\,1 \otimes \Phi((\Delta\boxtimes\alpha)(x))\\
& =\,(\id\otimes\Phi)\big(1\otimes(\Delta\boxtimes\alpha)(x)\big)\\
& =\,(\id\otimes\Phi)\big((\widehat{\Delta\boxtimes\alpha})
(\Delta\boxtimes\alpha)(x)\big).
\end{align*} 
	On the other hand, by (\ref{Phi2}), for any 
	$\hat x\in \LLL$ we have
	\begin{align*} 
\big((\Deltaop\otimes\id)\circ \Phi\big)(\hat x\otimes 1_{\DD\btp N})
& =\,\Deltaop(\hat x)\otimes 1_{\DD\vtp N}\\
& =\,(\id\otimes\Phi)(\widehat{\Delta\boxtimes\alpha})
(\hat x\otimes1_{\DD\btp N}). \qedhere
\end{align*} 
\end{proof}

\section{Examples}\label{sect5}
 In this section, we give some examples
 of amenable actions of discrete quantum groups 
 on von Neumann algebras.
 In Theorem \ref{relation}
 we showed that 
 the amenable quantum group $\G$ acts amenably on 
 any von Neumann algebras.
 Also, Proposition \ref{extension} shows that
 it is possible to get new amenable actions by 
 appropriate restrictions. Below we
 give more concrete examples of amenable actions.
 As an application of amenable actions, the action of any
 discrete group $G$ on $\ell^\infty(G)$ is 
 always amenable \cite[Remarques 3.7.(b)]{Del}.
 The next result is the noncommutative analogue of that.

\begin{proposition}\label{itself}
	Every discrete quantum group acts amenably on itself.
\end{proposition}
\begin{proof}
	Define the map $\Phi: \DD\vtp\DD\to \DD$ by 
	$\Phi(A)=(\id\otimes \varepsilon)(A)$, in which $\varepsilon$ 
	is the co-unit in $\DO$. Then 
	\[
	\Phi\circ\Delta = (\id\otimes\varepsilon)\Delta = \id.
	\]
	So the map $\Phi$ is a left inverse of 
	the co-multiplication $\Delta$ and
	therefore the map $E_\Delta:=\Delta\circ\Phi$ is 
	a conditional expectation 
	from $\DD\vtp\DD$ onto $\Delta(\DD)$.
	Moreover, for any $A\in\DD\vtp\DD$ we have 
	\begin{align*}
	 (\id\otimes E_\Delta)(\Delta\otimes\id)(A) & =\,
     (\id\otimes\Delta\circ\Phi)(\Delta\otimes\id)(A)\\
& =\,(\id\otimes\Delta)(\id\otimes\id\otimes\varepsilon)
(\Delta\otimes\id)(A)\\  
& =\,(\id\otimes\Delta)\Delta\big((\id\otimes \varepsilon)(A)\big)\\
& =\, (\Delta\otimes\id)\big((\Delta\circ\Phi)(A)\big)\\
& =\, (\Delta\otimes\id)E_\Delta(A).
\end{align*} 
Hence  
$E_\Delta:(\DD\vtp\DD, \Delta\otimes\id)\to
(\Delta(\DD), \Delta\otimes\id)$
is an equivariant conditional expectation.
\end{proof}

For a discrete quantum group $\G$ the restriction of 
the extended comultiplication
$\Delta^\ell$ to $\LLL$ provides an action
of $\G$ on the von Neumann
algebra $\LLL$.
we next prove in the
case of Kac algebras, amenability of the latter action 
is equivalent to amenability of the
Kac algebra $\G$.

\begin{theorem}\label{sufficient condition}
	Let $\G$ be a discrete Kac algebra. Then $\G$ is amenable 
	if and only if the canonical action
	$\Delta^\ell_{|_{\LLL}}:\G\curvearrowright\LLL$ is amenable. 
\end{theorem}

\begin{proof} 
Since $\G$ is a
Kac algebra, by \cite[Corollary 3.9]{I}, the
tracial Haar state $\hat\fee$ of the dual 
quantum group $\hat\G$ is invariant 
with respect to the action $\Delta^\ell_{|_{\LLL}}$. Hence by
Theorem \ref{relation}, amenability of the action
$\Delta^\ell_{|_{\LLL}}$ is
equivalent to amenability of $\G$.
\end{proof}

\begin{remark}
In \cite{Ja-inner}, Crann defined 
a notion of inner amenability for quantum groups as
the existence of an invariant state for
the canonical action $\Delta^\ell_{|_{\LLL}}:\G\curvearrowright\LLL$.
The same proof shows that Theorem \ref{sufficient condition}
holds for the inner amenable quantum group $\G$ in the sense of Crann.
\end{remark}
 In the next result, we state the noncommutative version
of Zimmer's classical result \cite[Theorem 5.2]{ZIM} 
that all Poisson boundaries are amenable $G$-space.
Let us first recall the definition of 
noncommutative Poisson boundaries in the sense of Izumi 
\cite{II}. Let $\mu\in\DO$ be a state. Recall in this case 
$\Phi_\mu(x)= (\mu\otimes\id)\Delta(x)$ is 
a unital, normal completely 
positive map on $\DD$.
The space of fixed point
$\h_\mu = \{ x\in\DD: \Phi_\mu(x)=x\}$ 
is a $w^*$-closed operator system in $\DD$.
There is a conditional expectation
$\E$ from $\DD$ onto $\h_\mu$. Then
the corresponding Choi--Effros product
induces the von Neumann algebraic structure 
on $\h_\mu$ \cite{Choi-Eff}. This von Neumann algebra 
is called {\it noncommutative Poisson boundary}
 with respect to $\mu$. For more details on
 noncommutative Poisson boundaries we refer the reader
 to \cite{KNR} and \cite{KNR2}.
By \cite[Proposition 2.1]{KNR2}, the restriction of $\Delta$
to $\h_\mu$ induces a left action $\Delta_\mu$ of $\G$ on the 
von Neumann algebra $\h_\mu$. We prove
this action is amenable.

\begin{theorem}\label{Poisson}
	Let $\G$ be a discrete quantum group and let $\mu\in \DO$
	be a state. The left
	action $\Delta_\mu$ of $\G$ on the 
	Poisson boundary $\h_\mu$ is amenable.
\end{theorem}

\begin{proof}
	The conditional expectation $\E:\DD\to\h_\mu$ is equivariant,
	see e.g. the proof of \cite[Proposition 2.1]{KNR2},
	and therefore the pair 
	$(\DD, \h_\mu)$ is amenable. Since by Proposition \ref{itself}
	the action of discrete quantum group
	on itself is amenable, it follows from
	Proposition \ref{extension} that
	the left action $\Delta_{\mu}$ 
	is amenable.
\end{proof}
\begin{remark}
In \cite{VV}, Vaes and Vergnioux introduced 
the amenable action of a discrete quantum group
on a unital C*-algebra. They proved that
the canonical C*-algebraic action of a universal 
discrete quantum group on its boundary is always 
amenable. Therefore the related crossed product
becomes nuclear. 
\end{remark}

\section{Amenable actions and crossed products: Kac algebra case}\label{sect6}

In this section we characterize amenability of actions
in term of von Neumann algebra crossed products.
Classically, the action $\alpha:G\curvearrowright X$
of a discrete group $G$ on a standard probability space
$(X, \nu)$ is amenable if and only if 
the crossed product $L^{\infty}(X, \nu)\ltimes_\alpha G$
is injective. Delaroche extended this
result to the action of locally compact groups on
arbitrary von Neumann algebras.
We prove a noncommutative version of this result
in the case of discrete Kac algebra actions 
on von Neumann algebras. This in particular generalizes
a part of Theorem 4.5 in \cite{Ruan} and also 
\cite[Corollary 3.17]{Reiji} which
establish the equivalence between amenability of a
discrete Kac algebra $\G$ and
injectivity of $\LLL$. 
The case of actions of general discrete quantum groups
on von Neumann algebras is discussed in next section.

\begin{lemma}\label{Kac}
	Let $\alpha:\G\curvearrowright N$ be an 
	action of a discrete Kac algbera $\G$ on 
	a von Neumann algebra
	$N$ and let $M$ be a von Neumann subalgebra of $N$ 
	which is invariant under
    $\alpha$. The following are equivalent:
	\begin{itemize}
		\item [1.] There is an equivariant conditional expectation
		            $P:(N, \alpha)\to (M, \alpha)$.
		\item [2.] There is a conditional expectation 
		           $E:N\ltimes_\alpha \G \to M\ltimes_\alpha \G$.
	\end{itemize}
\end{lemma}

\begin{proof}
	(1) $\Rightarrow $ (2): since $P:(N, \alpha)\to (M,\alpha)$
	 is an equivariant conditional expectation, it follows
	 $E=\id\otimes P$ is a conditional expectation 
	 from $B(\DT)\vtp N$ onto $B(\DT)\vtp M$ such that
	\[
	(\chi\otimes\id)(\id\otimes\alpha)E=(\id\otimes E)(\chi\otimes\id)(\id\otimes\alpha),
	\]
	where $\chi$ is the flip map.
	Then $\id\otimes E$ is a conditional 
	expectation from $\DD\vtp B(\DT)\vtp N$ 
    onto $\DD\vtp B(\DT)\vtp M$.  Recall the
    left action $\beta$ of $\G$ on $B(\DT)\vtp N$, defined in Theorem 
    \ref{characterization}. One can see that for any $y\in B(\DT)\vtp N$
    we have $\beta \circ E(y)=(\id\otimes E)\beta(y)$.
    In particular, if 
	$y\in (B(\DT)\vtp N)^\beta$ we have
	\[
	\beta \circ E(y)=(\id\otimes E)\beta(y)=
	(\id \otimes E)(1\otimes y)=1\otimes E(y),
	\]
	which implies $E(y)\in (B(\DT)\vtp M)^\beta$. Thus
	in view of Theorem \ref{characterization}, the restriction of 
	$E$ is a conditional expectation from 
	$N\ltimes_\alpha \G$ onto $M\ltimes_\alpha \G$.\\
	(2) $\Rightarrow $ (1): suppose 
	$E: N\ltimes_\alpha \G\to M\ltimes_\alpha \G$
	is the conditional expectation and 
	$\hat\fee$ is the tracial Haar state of the
	dual Kac algebra $\hat\G$.
	There is a canonical 
	conditional expectation $E_{\hat\fee}$ 
	from $M\ltimes_\alpha \G$
	onto $\alpha(M)$ defined by 
	\begin{equation}\label{canonical expectation}
	E_{\hat\fee}(x)=(\hat\fee\otimes\id)\hat\alpha(x),
	\,\,\,\,\,\,{\text {for all}}\, x\in M\ltimes_\alpha \G.
	\end{equation}
	We claim that 
	$E_{\hat\fee}\circ E: (N\ltimes_\alpha \G, \Delta\otimes\id)
	\to (\alpha(M), \Delta\otimes\id)$ is an
	equivariant conditional expectation.
	Since the left fundamental unitary
    $W$ lies in $\DD\vtp \LLL$, for all
    $z\in N\ltimes_\alpha \G$ we have
    \begin{align*}
(\id \otimes E)(\Delta \otimes \id)(z) 
& =\,
(\id \otimes E)(W_{12}^*z_{23}W_{12})\\
& =\, 
W_{12}^*(1 \otimes E(z))W_{12}\\
& =\,
(\Delta\otimes \id)(E(z)).
\end{align*} 
Therefore in order to conclude the claim, it is sufficient to show that 
the canonical conditional expectation 
$E_{\hat\fee}:(M\ltimes_\alpha \G, \Delta\otimes\id)
\to (\alpha(M), \Delta\otimes\id)$ is
equivariant. First, consider $\hat x\in \LLL$. Then 
    \begin{align*}
(\id \otimes E_{\hat\fee})(\Delta \otimes \id)(&\hat x\otimes 1) 
\, =\,
(\id \otimes E_{\hat\fee})(W_{12}^*(1\otimes \hat x\otimes 1) W_{12})\\
& =\, 
(\id \otimes \hat\fee\otimes \id)(\id\otimes \hat\alpha)(W_{12}^*(1\otimes \hat x\otimes 1) W_{12})\\
& =\,
(\id \otimes \hat\fee\otimes \id\otimes \id)(\id\otimes \Deltaop\otimes \id)(W_{12}^*(1\otimes \hat x\otimes 1) W_{12})\\
& =\, 
(\id \otimes \id \otimes \hat\fee\otimes \id)(\id\otimes \hat\Delta\otimes \id)(W_{12}^*(1\otimes \hat x\otimes 1) W_{12})\\
& =\,
(\id \otimes \id \otimes \hat\fee\otimes \id)(W_{13}^*(1\otimes 1\otimes \hat x\otimes 1) W_{13}).
\end{align*}
    Now consider a complete orthonormal 
    system $\{e_j\}_{j\in J}$. Then
    Similarly to the proof of 
	\cite[Corollary 3.17]{Reiji}, for 
	any normal states $f\in\DO$ and
	$\omega\in M_*$, and any vector state  
	$\omega_\xi\in\mathcal{T}(\DT)$ we get
\begin{align*}
\la \omega_\xi\otimes f\otimes \omega , (\id \otimes E_{\hat\fee})(\Delta
&\otimes \id)(\hat x\otimes 1) \rangle\,=\,f(1)\omega(1)
\la \omega_\xi \otimes \hat\fee , W^*(1\otimes \hat x) W \rangle\\
&=\,\la \hat\fee , (\omega_\xi\otimes \id)(W^*(1\otimes \hat x) W)\rangle  \\
&=\,\la \hat\fee , \sum_{j\in J}{{(\omega_{\xi,e_j}\otimes \id)(W)}^*
\hat x (\omega_{\xi,e_j}\otimes \id)(W)}\rangle\\
&=\,\sum_{j\in J}\la \hat\fee ,{{(\omega_{\xi,e_j}\otimes \id)(W)}^*
\hat x (\omega_{\xi,e_j}\otimes \id)(W)}\rangle\\
&=\,\sum_{j\in J}\la \hat\fee ,{{(\omega_{\xi,e_j}\otimes \id)(W)(\omega_{\xi,e_j}\otimes \id)(W)}^*\hat x}\rangle\\
&=\,\sum_{j\in J}\la \hat\fee ,{{(\omega_{\hat{J}\xi,\hat{J}e_j}\otimes \id)(W)^*(\omega_{\hat{J}\xi,\hat{J}e_p}\otimes \id)(W)}\hat x}\rangle\\
&=\,\omega_{\hat{J}\xi}(1)\hat\fee(\hat x)\\
&=\,\omega_\xi(1)\hat\fee(\hat x),
\end{align*}
where $\hat J$ is the modular conjugation for the
tracial Haar state $\hat\fee$.	
	So for any $\hat x\in\LLL$ we have  
	\begin{align*}
(\id \otimes E_{\hat\fee})(\Delta \otimes \id)(\hat x\otimes 1)
& =\,
\hat\fee(\hat x) 1\otimes 1\otimes 1\\
& =\, 
(\Delta\otimes \id)(\hat \fee\otimes \id \otimes \id)(\Deltaop(\hat x)
\otimes 1)\\
& =\,
(\Delta\otimes \id)(\hat \fee\otimes \id)
\big(\hat\alpha(\hat x\otimes 1)\big)\\
& =\,
(\Delta\otimes \id)E_{\hat\fee}(\hat x\otimes 1).
\end{align*}
	Hence for any $\hat x\in \LLL$ and $x\in M$
	we get
	\begin{align*}
(\id \otimes E_{\hat\fee})(\Delta \otimes \id)
\big((\hat x\otimes 1)\alpha(x)\big)
& =\, 
(\id \otimes E_{\hat\fee})\big((\Delta \otimes \id)(\hat x\otimes 1)
(\Delta \otimes \id)\alpha(x)\big)\\
& =\,
(\id \otimes E_{\hat\fee})\big((\Delta \otimes \id)(\hat x\otimes 1)
(\id \otimes \alpha)\alpha(x)\big)\\
& =\,
(\id \otimes E_{\hat\fee})\big((\Delta \otimes \id)(\hat x\otimes 1)\big)
(\id \otimes \alpha)\alpha(x)\\
& =\,
(\Delta \otimes \id)\big(E_{\hat\fee}(\hat x\otimes 1)\big)
(\Delta \otimes \id)\alpha(x)\\
& =\,
(\Delta \otimes \id)E_{\hat\fee}\big((\hat x\otimes 1)\alpha(x)\big).
\end{align*} 

Since the crossed product $M \ltimes_\alpha \G$
is generated by $\{\DDD\vtp {\bf{1}}, \alpha(M)\}$, 
it follows that 
the conditional expectation 
$E_{\hat\fee}$ is $(\Delta\otimes\id)$-equivariant 
which completes the proof of claim. 
Define the conditional expectation 
$P:=\alpha^{-1}\circ E_{\hat\fee} \circ E \circ \alpha$ from $N$
onto $M$. We show that $(\id\otimes P)\alpha\,=\,\alpha\circ P$.
Since $E_{\hat\fee}\circ E$ is equivariant
with respect to the action $\Delta\otimes\id$, 
for all $a\in N$ we have
\begin{align*}
(\id \otimes E_{\hat\fee}\circ E)(\id \otimes \alpha)\alpha(a) 
&\,=(\id \otimes E_{\hat\fee}\circ E)(\Delta \otimes \id)\alpha(a) \\
&\,=(\Delta \otimes \id)(E_{\hat\fee}\circ E)\alpha(a)\\
&\,=(\id \otimes \alpha)\circ (E_{\hat\fee}\circ E)(\alpha(a)),
\end{align*}
where in the last equality we use that
$(E_{\hat\fee}\circ E)\alpha(a)\in \alpha(M)$.
Now it follows 
\[
(\id\otimes P)\alpha = (\id \otimes \alpha^{-1})
(\id \otimes E_{\hat\fee}\circ E)(\id \otimes \alpha)\alpha = 
E_{\hat\fee}\circ E\circ \alpha = \alpha\circ P. \qedhere
\]
	
\end{proof}

The following is the noncommutative analogue of 
the main result of \cite{Del} for discrete Kac algebra actions:
(See \cite[Theorem 4.2]{Del}.)
\begin{theorem}\label{Del for Kac}
	Let $\alpha:\G\curvearrowright N$ be an action of a 
	discrete Kac algebra $\G$ on a von Neumann algebra 
	$N$. The following are equivalent:
	\begin{itemize}
	\item [1.] The action $\alpha$ is amenable.
	\item [2.] There is a conditional expectation from 
	           $(\DD\btp N)\ltimes_{\Delta\boxtimes\alpha}\G$
	           onto $({\bf 1}\btp N)\ltimes_{\Delta\boxtimes\alpha}\G$.
	\item [3.] For any extension $(M, \G, \beta)$ of 
	           $({\bf 1}\btp N, \G, \Delta\boxtimes\alpha)$, the pair 
	           $(M, {\bf 1}\btp N)$ is amenable.
	\end{itemize}
\end{theorem}

\begin{proof}
	$(1) \Rightarrow (2)$: suppose that $\alpha$ is amenable, then by
	Proposition \ref{amenable pair} the pair $(\DD\btp N, {\bf 1}\btp N)$
	is amenable which means there is an equivariant
	conditional expectation from 
	$(\DD\btp N, \Delta\boxtimes \alpha)$ onto 
	$({\bf 1}\btp N, \Delta\boxtimes \alpha)$. 
	Hence $(2)$ follows by Lemma \ref{Kac}.\\
	$(2) \Rightarrow (3)$: suppose that $(M, \G, \beta)$ is 
	an extension of $({\bf 1}\btp N, \G, \Delta\boxtimes\alpha)$,
	let $q$ be a conditional expectation 
	from $M$ onto ${\bf 1}\btp N$. Then
	$\id\otimes q$ is a conditional expectation  
	from $B(\DT)\vtp M$ onto $B(\DT)\vtp({\bf 1}\btp N)$,
	and thus Theorem \ref{isomorphism} yields 
	a conditional expectation
	\[
	E: (\DD\btp M)\ltimes_{\Delta\boxtimes\beta}\G\to 
	(\DD\btp({\bf 1}\btp N))\ltimes_{\Delta\boxtimes\beta}\G.
	\]
	Moreover by the assumption there is a conditional
	expectation from the crossed product 
	$(\DD\btp N)\ltimes_{\Delta\boxtimes\alpha}\G$
	onto $({\bf 1}\btp N)\ltimes_{\Delta\boxtimes\alpha}\G$.
	By the equality (\ref{associative}), it is equivalent to
	the existence of a conditional
	expectation $E_0$ from
	$\big(\DD\btp ({\bf 1}\btp N)\big)\ltimes_{\Delta\boxtimes\beta}\G$
	onto 
	$\big({\bf 1}\btp ({\bf 1}\btp N)\big)\ltimes_{\Delta\boxtimes\beta}\G$.
	By composing, we obtain the conditional expectation
	\[
	E_0\circ E:(\DD\btp M)\ltimes_{\Delta\boxtimes\beta}\G\to
	\big({\bf 1}\btp ({\bf 1}\btp N)\big)\ltimes_{\Delta\boxtimes\beta}\G.
	\]
	Hence from Lemma \ref{Kac}, we have an
	equivariant conditional expectation $Q$ from 
	$(\DD\btp M, \Delta\boxtimes\beta)$ onto 
	$\big({\bf 1}\btp ({\bf 1}\btp N), \Delta\boxtimes\beta\big)$. Since
	$(M,\G,\beta)$ is an extension of 
	$({\bf 1}\btp N,\G, \Delta\boxtimes\alpha)$, 
	the restriction of $Q$ to $1\btp M$ yields a
    conditional expectation
	$Q_0:\beta(M)\cong{\bf 1}\btp M\to 
	{\bf 1}\btp ({\bf 1}\btp N)\cong\beta({\bf 1}\btp N)$
	such that 
	$(\id\otimes Q_0)(\Delta\boxtimes\beta)=(\Delta\boxtimes\beta)\circ Q_0$.
	Hence for any $a\in M$ we have
	\[
	(\id\otimes Q_0)(\id\otimes\beta)\beta(a)=
	(\id\otimes \beta)Q_0\big(\beta(a)\big).
	\]
	Now define a conditional expectation 
	$P:=\beta^{-1}\circ Q_0\circ \beta$ from
	$M$ onto ${\bf 1}\btp N$. Then for all $a\in M$ 
	we have
	\begin{align*}
		(\id\otimes P)\beta(a)&=\,
		(\id\otimes \beta^{-1})(\id\otimes Q_0)(\id\otimes\beta)\beta(a)\\
		&\,=(\id\otimes \beta^{-1})(\id\otimes\beta)Q_0\big(\beta(a)\big)\\
		&\,=\beta\circ P(a),
	\end{align*}
	which implies 
	$P:(M, \beta)\to ({\bf 1}\btp N, \Delta\boxtimes\alpha)$ 
	is an equivariant 
	conditional expectation. Hence the pair
	$(M, {\bf 1}\btp N)$ is amenable.
	\newline
	$(3) \Rightarrow (1)$: consider the canonical
	extension $(\DD\btp N, \G, \Delta\boxtimes\alpha)$ of the triple
	$({\bf 1}\btp N, \G, \Delta\boxtimes\alpha)$. 
	Then by the assumption the
	pair $(\DD\btp N, {\bf 1}\btp N)$ must be amenable.  
	Hence the action $\alpha$ is 
	amenable by Proposition \ref{amenable pair}.
\end{proof}

\begin{theorem}\label{corollary}
	Let $\alpha:\G\curvearrowright N$ be an action of 
	a discrete Kac algebra $\G$ on a von Neuamnn algebra 
	$N$. Then the following are equivalent:
	\begin{itemize}
		\item [1.] The action $\alpha$ is amenable.
		\item [2.] There is a conditional expectation from 
	           $B(\DT)\vtp N$ onto $N\ltimes_\alpha\G$.
	\end{itemize}
\end{theorem}
\begin{proof}
	By Theorem \ref{isomorphism}, there is an isomorphism
	from $(\DD\btp N)\ltimes_{\Delta\boxtimes\alpha}\G$
	onto $B(\DT)\vtp N$ which maps 
	$({\bf 1}\btp N)\ltimes_{\Delta\boxtimes\alpha}\G$ onto
	$N\ltimes_\alpha\G$. So the Theorem follows by the
	equivalence of (1) and (2) in 
	Theorem \ref{Del for Kac}.	
\end{proof}

\begin{corollary}\label{injectivity of crossed product}
	Let $\alpha:\G\curvearrowright N$ be an action 
	of a discrete Kac algebra 
	$\G$ on a von Neuamnn algebra $N$. 
	Then the following are equivalent:
	\begin{itemize}
		\item [1.] The von Neumann algebra $N$ is injective and the
		           action $\alpha$ is amenable.
		\item [2.] The crossed product $N\ltimes_\alpha\G$ is injective.
	\end{itemize}
\end{corollary}
\begin{proof}
	$(1)\Rightarrow(2)$: if $N$ is injective then so is
	$B(\DT)\vtp N$. If $\alpha$ is amenable, 
	Theorem \ref{corollary} yields
    a conditional expectation from 
	$B(\DT)\vtp N$ onto the crossed product
	$N\ltimes_\alpha\G$. Since $B(\DT)\vtp N$ is
	injective, $N\ltimes_\alpha\G$ is also injective.
	\newline
	$(2)\Rightarrow(1)$: since the crossed product 
	$N\ltimes_\alpha\G$ is injective, there is a conditional expectation
	from $B(\DT)\vtp N$ onto $N\ltimes_\alpha\G$. Therefore 
	by Theorem \ref{corollary}, the action $\alpha$ is amenable. 
	Moreover since there is always the canonical
	conditional expectation from $N\ltimes_\alpha\G$
	on $\alpha(N)$, it follows that $\alpha(N)$ 
	and equivalently 
	$N$, is injective.	
\end{proof}
\section{Amenable actions and crossed products: general case}\label{sect7}

	In this section, we generalize the duality of Corollary
	\ref{corollary} to the setting of discrete quantum group
	actions. For this end, we basically need to show 
	Lemma \ref{Kac} for general 
	discrete quantum groups. 
	Recall that in the proof of the implication $(2)$ to $(1)$
	of Lemma \ref{Kac}, we construct an equivariant conditional 
	expectation from  $(\alpha(N),\Delta\otimes\id)$
	onto $(\alpha(M),\Delta\otimes\id)$ by composing 
	the restriction $E_{|_{\alpha(N)}}$ with
	the canonical conditional expectation $E_{\hat\fee}$.
	In the case of
	discrete Kac algebras, $E_{\hat\fee}$ is
	automatically equivariant
	with respect to the action
	$\Delta\otimes\id$.
	But this is no longer the case in the
	general setting of discrete quantum group actions, 
	since the Haar state $\hat\fee$ is 
	not a trace. To overcome this issue, we impose
	an extra assumption on the conditional
	expectation $E:N\ltimes_\alpha\G\to M\ltimes_\alpha\G$
	to be equivariant with respect to the dual action $\hat\alpha$.
	This would imply that $E$ maps $\alpha(N)$ onto $\alpha(M)$,
	hence use of the canonical conditional expectation
	$E_{\hat\fee}$ is no longer necessary. This inspired 
	by the work of Crann and Neufang in \cite{JN}, where
	they proved a characterization of amenability of the general
	locally compact quantum group $\G$ in terms of covariant 
	injectivity of the dual von Neumann algebra $\LLL$.

\begin{lemma}
	Let $\G$ be a discrete quantum group. Then for any $y\in B(\DT)$
	we have
	\[
	(\id\otimes\Delta^r)\Deltaop(y)=(\Deltaop\otimes\id)\Delta^r(y).
	\]
\end{lemma}
\begin{proof}
	Let $x\in\DD$ and $\hat x\in \LLL$. Since the
	fundamental unitaries ${\hat{W}}^{\text {op}}$ 
	and $V$ lie in $\LLL\vtp\DD'$ and 
	$\LLL'\vtp\DD$, respectively, we have
	\[
	\Deltaop(x)=1\otimes x\hspace{.5cm}{\text{and}}
	\hspace{.5cm}\Delta^r(\hat x)=\hat x\otimes 1.
	\]
	Therefore
	\[
	(\id\otimes\Delta^r)\Deltaop(x)=(\id\otimes\Delta^r)(1\otimes x)
	=1\otimes\Delta^r(x)=(\Deltaop\otimes\id)\Delta^r(x),
	\]
	and
	\[
	(\id\otimes\Delta^r)\Deltaop(\hat x)=\Deltaop(\hat x)\otimes 1
	=(\Deltaop\otimes\id)(\hat x\otimes 1)=
	(\Deltaop\otimes\id)\Delta^r(\hat x).
	\]
	Since the co-multiplications $\Delta^r$ and 
	$\Deltaop$ are homomorphisms, and the linear span of 
	$\{x\hat x: x\in\DD, \hat x \in\LLL\}$ is weak* dense 
	in $B(\DT)$ \cite[Proposition 2.5]{VVAN}, 
	we obtain the desired equality on $B(\DT)$.
\end{proof}
In the following we use the same idea as
\cite[Proposition 4.2]{JN2} to show an automatic equivariant
property with respect to the dual action.
\begin{proposition}\label{Delta-Deltaop}
	Let $\G$ be a discrete quantum group and let $N$ be a von Neumann
	algebra. Then any $(\Delta^r\otimes\id)$-equivariant map
	on $B(\DT)\vtp N$ is automatically 
	$(\Deltaop\otimes\id)$-equivariant.
\end{proposition}
\begin{proof}
    Let $\Phi$ be an equivariant map on
    $(B(\DT)\vtp N, \Delta^r\otimes\id)$.
	Consider normal states $\tau,\omega\in\T(\DT)$,
	$f\in \DO$ and $g\in N_*$. Then for any
	$x\in B(\DT)\vtp N$ we have
	\begin{align*}
\la f\otimes\tau\otimes &\omega\otimes g, (\id\otimes\Delta^r\otimes\id)
	(\id\otimes \Phi)(\Deltaop\otimes\id)(x)\rangle\\
&=\,\la \tau \otimes\omega\otimes g, (\Delta^r\otimes\id)
	\Phi\big((f\otimes\id\otimes\id)(\Deltaop\otimes\id)(x)\big)\rangle\\
&=\,\la \tau \otimes\omega\otimes g,
	(\id\otimes\Phi)\big[(\Delta^r\otimes\id)\big((f\otimes\id\otimes\id)(\Deltaop\otimes\id)(x)\big)\big]\rangle\\
&=\,\la \omega\otimes g, \Phi\big[(\tau \otimes\id\otimes\id)
(\Delta^r\otimes\id)\big((f\otimes\id\otimes\id)(\Deltaop\otimes\id)(x)
\big)\big]\rangle\\
&=\,\la \omega\otimes g, \Phi\big[(f \otimes \tau \otimes\id\otimes\id)
(\id\otimes\Delta^r\otimes\id)\big((\Deltaop\otimes\id)(x)
\big)\big]\rangle\\
&=\,\la \omega\otimes g, \Phi\big[(f \otimes \tau \otimes\id\otimes\id)
(\Deltaop\otimes\id\otimes\id)\big((\Delta^r\otimes\id)(x)
\big)\big]\rangle\\
&=\,\la \omega\otimes g, (f \otimes \tau \otimes\id\otimes\id)
(\Deltaop\otimes\id\otimes\id)\Phi\big((\Delta^r\otimes\id)(x)
\big)\rangle\\
&=\,\la f \otimes \tau\otimes\omega\otimes g, (\Deltaop\otimes\id\otimes\id)
(\Delta^r\otimes\id)\big(\Phi(x)\big)\rangle\\
&=\,\la f \otimes \tau\otimes\omega\otimes g, 
(\id\otimes\Delta^r\otimes\id)(\Deltaop\otimes\id)\big(\Phi(x)\big)\rangle.
\end{align*}
Since $\{(\tau\otimes\omega)\Delta^r: \tau,\omega\in\T(\DT)\}$
spans a dense subset of $\T(\DT)$, see (\ref{traceclass}), it follows
\[
(\id\otimes\Phi)\big((\Deltaop\otimes\id)(x)\big)=
	(\Deltaop\otimes\id)\Phi(x).
\qedhere
\]
\end{proof}
\begin{corollary}\label{dual equivariant}
	Let $\alpha:\G\curvearrowright N$ be an action of a 
	discrete quantum group $\G$ on a
	von Neumann algebra
	$N$ and let $M$ be a von Neumann subalgebra of 
	$N$ which is invariant under
    $\alpha$. If 
    $E:(N\ltimes_\alpha \G,\Delta^r\otimes\id)\to(M\ltimes_\alpha \G,\Delta^r\otimes\id)$
    is an equivariant conditional expectation, then
    $E$ is equivariant with respect to the dual action $\hat \alpha$.
    \end{corollary}
    \begin{proof}
    	Note that the dual action $\hat\alpha$ is the restriction 
    	of $\Deltaop\otimes\id$
    	to the crossed product 
    	$N\ltimes_\alpha \G\subseteq B(\DT)\vtp N$. Hence
    	Proposition \ref{Delta-Deltaop} implies that
    	the conditional expectation $E$ is equivariant with respect
    	to $\hat\alpha$.
    \end{proof}
\begin{lemma}\label{crossed product}
	Let $\alpha:\G\curvearrowright N$ be an action of a 
	discrete quantum group $\G$ on a
	von Neumann algebra
	$N$ and let $M$ be a von Neumann subalgebra of 
	$N$ which is invariant under
    $\alpha$. The following are equivalent:
	\begin{itemize}
		\item [1.] There is an equivariant 
		          conditional expectation $P:(N, \alpha)\to (M, \alpha)$.
		\item [2.] There is an equivariant 
		           conditional expectation 
       \[
	    E:(N\ltimes_\alpha \G, \hat\alpha)\to (M\ltimes_\alpha \G,\hat\alpha). 
	   \]
    \end{itemize}
\end{lemma}

\begin{proof}
(1) $\Rightarrow $ (2): similarly as in the proof of Lemma \ref{Kac},
we see that the restriction of 
$\id\otimes P: B(\DT)\vtp N\to B(\DT)\vtp M$ to the crossed product
$N\ltimes_\alpha \G$ yields a conditional
expectation $E$ from $N\ltimes_\alpha \G$ onto $M\ltimes_\alpha \G$.
It is easy to see that $E$ is $(\Delta^r\otimes\id)$-equivariant.
Thanks to Corollary \ref{dual equivariant} the conditional
expectation $E$ is equivariant
with respect to the dual action $\hat\alpha$.
\newline
(2) $\Rightarrow $ (1): suppose that
$E:(N\ltimes_\alpha \G,\hat\alpha)\to 
	(M\ltimes_\alpha \G,\hat\alpha)$ is
	an equivairant conditional expectation. Then for all
	$x\in N$ we have
	\begin{align*}	 
	 1_\LLL\otimes E(\alpha(x)) & =\,
	 (\id\otimes E)\big(1_\LLL\otimes \alpha(x)\big)\\
	 & =\,(\id\otimes E)\circ\hat{\alpha}(\alpha(x))\\
	 & =\,\hat{\alpha}\circ E(\alpha(x)).
    \end{align*} 
	It follows that $E(\alpha(x))$ is in the
	fixed point algebra of the dual action 
	$\hat\alpha$ on
	$M\ltimes_{\alpha}\G$. Hence
	$E(\alpha(N))\subseteq\alpha(M)$.
Now define the conditional expectation 
$P:=\alpha^{-1}\circ E\circ \alpha$ from
$N$ onto $M$. Similarly to the proof
of Lemma \ref{Kac}, we show that 
$(\id\otimes P)\alpha\,=\,\alpha\circ P$.
Since the fundamental unitary $W$ lies in
$\DD\vtp\LLL$, it follows that  the condtional 
expectation $E$ is $(\Delta\otimes\id)$-equivariant.
Now for any $x\in N$ we have
\begin{align*}
(\id \otimes E)(\id \otimes \alpha)\alpha(x) 
&\,=(\id \otimes E)(\Delta \otimes \id)\alpha(x) \\
&\,=(\Delta \otimes \id)E(\alpha(x))\\
&\,=(\id \otimes \alpha)E(\alpha(x)),
\end{align*}
where in the last equality we use that
$E(\alpha(x))\in \alpha(M)$.
Now it follows 
\[
(\id\otimes P)\alpha = (\id \otimes \alpha^{-1})
(\id \otimes E)(\id \otimes \alpha)\alpha = 
E\circ \alpha = \alpha\circ P. \qedhere
\]

\end{proof}

The equivalence of (1) and (3) in the following
result is a noncommutative analogue of 
Zimmer's classical result
\cite[Theorem. 2.1]{ZIM2}.

\begin{theorem}\label{characterization1}
	Let $\alpha:\G\curvearrowright N$ be an action of a 
	discrete quantum group $\G$ on a von Neumann algebra
	$N$. The following are equivalent:
	\begin{itemize}
		\item[1.] The action $\alpha$ is amenable.
		\item[2.] There is an equivariant
		          conditional expectation
	     \[
		 \hspace{1cm}E:\big((\DD\btp N)\ltimes_{\Delta\boxtimes\alpha}\G,
		          \widehat{\Delta\boxtimes\alpha}\big)\to 
		          \big(({\bf 1}\btp N)\ltimes_{\Delta\boxtimes \alpha} \G,
		          \widehat{\Delta\boxtimes\alpha}\big).
		 \]
		\item[3.] There is an equivariant conditional expectation
		\[
		\hspace{1cm}E:\big(B(\DT)\vtp N, \Deltaop\otimes\id)\to
		\big(N\ltimes_\alpha \G, \hat\alpha\big).
		\]
	\end{itemize}
\end{theorem}
\begin{proof}
	By Theorem \ref{isomorphism}, the statements $(2)$ and $(3)$ are 
	equivalent. To conclude $(1)$ and $(2)$, 
	thanks to Proposition \ref{amenable pair} 
	amenability of $\alpha$
	is equivalent to amenability of the pair
	$(\DD\btp N,{\bf 1}\btp N)$ which by Lemma 
	\ref{crossed product} is equivalent to the existence of a
	$\widehat{\Delta\boxtimes \alpha}$-equivariant 
	conditional expectation $E$ from 
	$(\DD\btp N)\ltimes_{\Delta\boxtimes \alpha}\G$
	onto $({\bf 1}\btp N)\ltimes_{\Delta\boxtimes \alpha}\G$
	\end{proof}

\begin{remark}
	Since the trivial action $\text{tr}:\G\curvearrowright\Bbb C$
	is amenable if and only if $\G$ is amenable, and 
	$\Bbb C\ltimes_\text{tr}\G=\LLL$, the equivalence of $(1)$ and
	$(3)$ in Theorem \ref{characterization1} in fact gives a
	generalization of the main result of \cite{JN}.
\end{remark}

	Suppose that $\beta:\G\curvearrowright K$ is 
	an action of a discrete quantum group $\G$ 
	on a von Neumann algebra $K$. We say that
	$K$ is $\G$-{\emph{injective}} if for every
	unital completely isometric equivariant map
	$\iota:(M, \alpha_1)\to(N, \alpha_2)$ and
	every unital completely positive equivariant map
	$\Psi:(M, \alpha_1)\to(K,\beta)$ there is
	a unital completely positive equivariant map
	$\overline\Psi:(N, \alpha_2)\to(K,\beta)$ 
	such that $\overline\Psi\circ \iota=\Psi$.
\begin{corollary}\label{cor}
	Let $\alpha:\G\curvearrowright N$ be 
	an action of a discrete quantum group $\G$ 
	on a von Neumann algebra $N$. 
	The following are equivalent:
	\begin{itemize}
		\item [1.] The von Neumann algebra $N$ is 
		           injective and the action $\alpha$ is amenable.
		\item [2.] The crossed product $N\ltimes_\alpha\G$
		           is $\hat\G$-injective.
	\end{itemize}
\end{corollary}
\begin{proof}
(1)$\Rightarrow$(2): the proof is similar to the proof of 
Corollary \ref{injectivity of crossed product}, only that
we use Theorem \ref{characterization1} 
instead of Theorem \ref{corollary}.
\newline
(2)$\Rightarrow$(1): since the crossed product $N\ltimes_\alpha\G$
is $\hat\G$-injective, the identity map on $N\ltimes_\alpha\G$
can be extended to an equivariant conditional
expectation from $(B(\DT)\vtp N, \Deltaop\otimes\id)$
onto $(N\ltimes_\alpha\G, \hat\alpha)$. Hence by
Theorem \ref{characterization1}, the action $\alpha$ is amenable.
Moreover there is always the canonical conditional 
expectation from $N\ltimes_\alpha\G$ onto $\alpha(N)$, it follows
that $N$ is injective.
\end{proof}

\begin{corollary}[\cite{KNR2}, Corollary 2.5]
	Let $\G$ be a discrete quantum group and let $\mu\in\DO$ 
	be a state. The von Neumann algebra crossed product
	$\h_\mu\ltimes_{\Delta_\mu}\G$ is injective.
\end{corollary}
\begin{proof}
	By Theorem \ref{Poisson} the action of $\G$ on its Poisson 
	boundaries is always amenable, and therefore
	the result follows by Corollary \ref{cor}. 
\end{proof}

\begin{remark}
	Crann and Kalantar informed us 
	in a recent unpublished 
	paper they have independently 
	defined a notion of Zimmer amenability
	in the setting of actions of 
	locally compact quantum groups
	on von Neumann algebras, where they 
	used a homological approach.
	But their definition is equivalent 
	to Definition \ref{amenability} in
	the case of discrete quantum groups. 
	They have obtained a similar
	result as Corollary \ref{cor} 
	in that general context. 
\end{remark}

\end{document}